\documentclass[10pt]{article}

\usepackage{amsmath,amssymb,amsthm,mathtools}
\usepackage{geometry}
\usepackage{enumitem}
\usepackage{hyperref}
\usepackage{mathrsfs}
\usepackage{authblk}
\newtheorem{theorem}{Theorem}[section]
\newtheorem{lemma}[theorem]{Lemma}
\newtheorem{proposition}[theorem]{Proposition}
\newtheorem{conjecture}[theorem]{Conjecture}
\newtheorem{remark}[theorem]{Remark}

\DeclareMathOperator{\rank}{rank}
\DeclareMathOperator{\Gl}{GL}

\title{A \(3\times 3\) counterexample to Lin and Wimmer's rank-minimization conjecture associated with Roth's similarity theorem}
\author[1]{Shuo Shi}
\author[1]{Juan Zhang}
\author[2]{Yun Zhang\thanks{Corresponding author. E-mail: zhangyunmaths@163.com}}
\affil[1]{School of Mathematics and Computational Science, Hunan Key Laboratory for Computation and Simulation in Science and Engineering, Key Laboratory of Intelligent Computing and Information Processing of Ministry of Education,
Xiangtan University, Xiangtan, Hunan,  411105, China}
\affil[2]{School of Mathematics and Statistics, Huaibei Normal University, Huaibei, 235000, China}
\date{}

\begin{document}

\maketitle

\begin{abstract}
We give a \(3\times 3\) counterexample, valid over every field, to a rank-minimization conjecture of Lin and Wimmer (Bull. Aust. Math. Soc., 84 (3) (2011), 441--443) related to Roth's similarity theorem for the Sylvester matrix equation.
We also prove that, over the complex field, no counterexample can occur when one of the two matrix sizes is less than \(3\).
Hence the example is dimensionally minimal over the complex field.
~
\end{abstract}

\section{Introduction}

Let \(K\) be a field.
For a positive integer \(r\), write
\[
\Gl(r,K)=\{M\in K^{r\times r}:\det M\neq 0\}.
\]
When the ground field is clear, we simply write \(\Gl(r)\).

Let
\[
A\in K^{m\times m},~
B\in K^{n\times n},~
C\in K^{m\times n}.
\]
For the solvability of the Sylvester equation
\begin{equation}\label{eq:Sylvester}
AX-XB=C, 
\end{equation}
we have the following results.

\begin{theorem}\label{thm:Roth}{\rm \cite[Roth's similarity theorem ]{Roth1952,zhan2013}}
Consider the Sylvester matrix equation \eqref{eq:Sylvester}. Then the equation \eqref{eq:Sylvester} is solvable with \(X\in K^{m\times n}\) if and only if there exists \(P\in\Gl(m+n,K)\) such that
\[
P
\begin{pmatrix}
A&C\\
0&B
\end{pmatrix}
=
\begin{pmatrix}
A&0\\
0&B
\end{pmatrix}
P.
\]
Moreover, over the complex field \(\mathbb C\), the equation \eqref{eq:Sylvester} has a unique solution for every \(C\) if and only if \(A\) and \(B\) have no common eigenvalue.
\end{theorem}

Roth's similarity theorem and its variants have been studied in several settings, including commutative rings and related algebraic contexts; see, for instance, \cite{Gustafson1979,Hartwig1977,WangGao2025}.
A different but closely related formulation is Roth's equivalence theorem for the generalized Sylvester matrix equation
\[
AX-YB=C,
\]
where the two unknowns \(X\) and \(Y\) are independent.
Lin and Wimmer \cite{LinWimmer2011} showed that Roth's equivalence theorem is a special case of a rank-minimization theorem, that is 
\begin{align}\label{eq:LWequivalence}
\min_{X,~Y\in K^{m\times n}}
\rank(AX-YB-C)
=
\min_{P,~Q\in\Gl(m+n,K)}
\rank\left(
P
\begin{pmatrix}
A&C\\
0&B
\end{pmatrix}
-
\begin{pmatrix}
A&0\\
0&B
\end{pmatrix}
Q
\right).
\end{align}
Ito and Wimmer \cite{ItoWimmer2013} later extended this direction to generalized Sylvester equations over B\'ezout domains.
In their note, Lin and Wimmer proposed the following rank-minimization analogue of Roth's similarity theorem.

\begin{conjecture}\label{con:LinWimmer}
Let \(K\) be a field, and let
$
A\in K^{m\times m},~
B\in K^{n\times n},~
C\in K^{m\times n}.
$
Then
\begin{align}\label{eq:LWconjecture}
\min_{X\in K^{m\times n}}
\rank(AX-XB-C)
=
\min_{P\in\Gl(m+n,K)}
\rank\left(
P
\begin{pmatrix}
A&C\\
0&B
\end{pmatrix}
-
\begin{pmatrix}
A&0\\
0&B
\end{pmatrix}
P
\right).
\end{align}
\end{conjecture}

Ferrante and Wimmer \cite{FerranteWimmer2013} subsequently studied this conjecture in detail.
They gave a criterion for equality; see \cite[Lemma~1.4]{FerranteWimmer2013}.
More importantly for the present note, they proved that \eqref{eq:LWconjecture} holds over the complex field \(\mathbb C\) under some spectral hypotheses \cite[Theorem~1.5]{FerranteWimmer2013}.

The purpose of this note is to show that Conjecture \ref{con:LinWimmer} is false without additional spectral hypotheses.
We give an explicit \(3\times 3\) counterexample.
We further prove that, over \(\mathbb C\), this \(3\times3\) size is minimal: if \(\min\{m,n\}<3\), then identity  \eqref{eq:LWconjecture} holds for every \(C\).

\section{Background and quoted results}\label{sec:background}

For clarity, we introduce the notation and the known results that will be used below.  
Set
\[
M_C=\begin{pmatrix}A&C\\0&B\end{pmatrix},
~
M_0=\begin{pmatrix}A&0\\0&B\end{pmatrix}
\]
and define
\[
\alpha(A,B,C)=
\min_{X\in K^{m\times n}} \rank(AX-XB-C),
\]
\[
\gamma(A,B,C)=
\min_{P\in\Gl(m+n,K)} \rank(PM_C-M_0P).
\]
Thus Conjecture~\ref{con:LinWimmer} asserts that
\[
\alpha(A,B,C)=\gamma(A,B,C)
\]
for all triples \((A,B,C)\).

The following elementary comparison is the starting point of the problem.  
It is also in \cite[Lemma~1.4~(i)]{FerranteWimmer2013}.

\begin{proposition}\label{prop:basic-comparison}
For every field \(K\) and every triple \((A,B,C)\),
\[
\gamma(A,B,C)\le \alpha(A,B,C).
\]
Moreover,
\[
\gamma(A,B,C)=0
~\Longleftrightarrow~
\alpha(A,B,C)=0.
\]
\end{proposition}

\begin{remark}
The equivalence of the zero cases follows from Roth's similarity theorem: \(\alpha(A,B,C)=0\) means that the Sylvester matrix equation \(AX-XB=C\) is solvable, and Roth's theorem is precisely the assertion that this is equivalent to the similarity of \(M_C\) and \(M_0\), i.e. to the existence of \(P\in\Gl(m+n,K)\) such that \(PM_C-M_0P=0\).  
This is exactly \(\gamma(A,B,C)=0\).
\end{remark}

We next recall the spectral terminology.  
Let \(T\in\mathbb C^{r\times r}\), and let \(\lambda\) be an eigenvalue of \(T\).  
The eigenvalue \(\lambda\) is called \emph{semisimple} for \(T\) if its geometric multiplicity equals its algebraic multiplicity.  
Equivalently, in the Jordan form of \(T\), every Jordan block corresponding to \(\lambda\) has size \(1\).  
The eigenvalue \(\lambda\) is called \emph{nonderogatory} for \(T\) if its geometric multiplicity is $1$, or, equivalently, if there is exactly one Jordan block corresponding to \(\lambda\).  
For example,
\[
J_1(\lambda)\oplus J_1(\lambda)
\]
is semisimple but not nonderogatory, while
\[
J_2(\lambda)
\]
is nonderogatory but not semisimple.  The smallest Jordan configuration which is neither semisimple nor nonderogatory is
\[
J_2(\lambda)\oplus J_1(\lambda).
\]

Ferrante and Wimmer proved the following sufficient conditions for Conjecture~\ref{con:LinWimmer} to hold.

\begin{theorem}{\rm \cite[Theorem~1.5]{FerranteWimmer2013}}\label{thm:FW-positive}
Let \(A\in\mathbb C^{m\times m}\), \(B\in\mathbb C^{n\times n}\), and \(C\in\mathbb C^{m\times n}\).  Suppose that one of the following two conditions holds:
\begin{enumerate}[label=\textup{(\roman*)}]
\item Each common eigenvalue of \(A\) and \(B\) is nonderogatory for at least one of the two matrices \(A\) and \(B\).
\item Each common eigenvalue of \(A\) and \(B\) is semisimple both for \(A\) and for \(B\).
\end{enumerate}
Then
$
\alpha(A,B,C)=\gamma(A,B,C)
$
for all \(C\in\mathbb C^{m\times n}\).
\end{theorem}

\begin{remark}
When $A$ and $B$ have no common eigenvalue, we have $\alpha(A,B,C)=\gamma(A,B,C)=0$, by Theorem \ref{thm:Roth} and Proposition \ref{prop:basic-comparison}.
This is why only the case of common eigenvalues is considered in \cite{FerranteWimmer2013}.
Therefore, unless otherwise stated, only the case where \(A\) and \(B\) share common eigenvalues is considered in this paper.
\end{remark}

\section{The counterexample}

We consider the following matrices over an arbitrary field \(K\):
\[
A=B=
\begin{pmatrix}
0&1&0\\
0&0&0\\
0&0&0
\end{pmatrix},
~
C=
\begin{pmatrix}
1&0&0\\
0&0&0\\
0&0&1
\end{pmatrix}.
\]
Equivalently,
\[
A=B=J_2(0)\oplus J_1(0).
\]
Thus the only eigenvalue of \(A\) and \(B\) is \(0\).  
Its algebraic multiplicity is \(3\), while its geometric multiplicity is \(2\).  
Hence it is not semisimple.  
It is also not nonderogatory, because there are two Jordan blocks associated with the same eigenvalue \(0\).  
Therefore this example lies outside both alternatives in Theorem~\ref{thm:FW-positive}.

We first compute the left-hand side of \eqref{eq:LWconjecture}.
Let
\[
X=
\begin{pmatrix}
x_{11}&x_{12}&x_{13}\\
x_{21}&x_{22}&x_{23}\\
x_{31}&x_{32}&x_{33}
\end{pmatrix}
\in K^{3\times 3}.
\]
Then
\[
AX-XB-C
=
\begin{pmatrix}
x_{21}-1&x_{22}-x_{11}&x_{23}\\
0&-x_{21}&0\\
0&-x_{31}&-1
\end{pmatrix}.
\]
Denote the three columns of this matrix by \(v_1,v_2,v_3\). 
Then
\[
v_1=
\begin{pmatrix}
x_{21}-1\\
0\\
0
\end{pmatrix},
~
v_2=
\begin{pmatrix}
x_{22}-x_{11}\\
-x_{21}\\
-x_{31}
\end{pmatrix},
~
v_3=
\begin{pmatrix}
x_{23}\\
0\\
-1
\end{pmatrix}.
\]

We claim that
\[
\rank(AX-XB-C)\geq 2
\]
for every \(X\in K^{3\times 3}\).
Indeed, suppose for contradiction that
\[
\rank(AX-XB-C)\leq 1.
\]
Since the third component of \(v_3\) is \(-1\neq 0\), the vector \(v_3\) is nonzero.
Hence every column must be a scalar multiple of \(v_3\).

In particular, \(v_1\) must be a scalar multiple of \(v_3\).
Since \(v_1\) has third component \(0\), whereas \(v_3\) has third component \(-1\neq 0\), this is possible only if \(v_1=0\).
Thus
\[
x_{21}=1.
\]
But then
\[
v_2=
\begin{pmatrix}
x_{22}-x_{11}\\
-1\\
-x_{31}
\end{pmatrix}.
\]
Its second component is \(-1\neq 0\), while every scalar multiple of \(v_3\) has second component \(0\).
This contradicts the assumption that all columns are scalar multiples of \(v_3\).
Therefore
\[
\rank(AX-XB-C)\geq 2
\]
for every \(X\in K^{3\times 3}\).

On the other hand, by taking \(X=0\), we obtain
\[
AX-XB-C=-C,
\]
and
\[
\rank(-C)=\rank(C)=2.
\]
Consequently,
\begin{equation}\label{eq:LHSvalue}
\min_{X\in K^{3\times 3}}
\rank(AX-XB-C)
=
2.
\end{equation}

We now compute the right-hand side of \eqref{eq:LWconjecture}.
Define
\[
P_0=
\begin{pmatrix}
1&0&0&0&0&0\\
0&1&0&1&0&0\\
0&0&0&0&1&0\\
0&0&1&0&0&0\\
0&0&0&0&0&1\\
0&1&0&0&0&0
\end{pmatrix}.
\]
A direct computation gives
\[
\det(P_0)=1,
\]
so \(P_0\in\Gl(6,K)\).
Moreover,
\[
P_0M_C-M_0P_0
=
\begin{pmatrix}
0&0&0&0&0&0\\
0&0&0&0&1&0\\
0&0&0&0&0&0\\
0&0&0&0&0&0\\
0&0&0&0&0&0\\
0&0&0&0&0&0
\end{pmatrix},
\]
which has rank \(1\).
Hence
\begin{equation}\label{eq:RHSatmost1}
\min_{P\in\Gl(6,K)}
\rank(PM_C-M_0P)
\leq 1.
\end{equation}

We next show that the minimum in \eqref{eq:RHSatmost1} is not \(0\).
If it were \(0\), then there would exist \(P\in\Gl(6,K)\) such that
\[
PM_C=M_0P.
\]
Equivalently,
\[
P
\begin{pmatrix}
A&C\\
0&B
\end{pmatrix}
=
\begin{pmatrix}
A&0\\
0&B
\end{pmatrix}
P.
\]
By Roth's similarity theorem, the Sylvester matrix equation
\[
AX-XB=C
\]
would then have a solution \(X\in K^{3\times 3}\).
This would imply
\[
AX-XB-C=0
\]
for some \(X\), contradicting \eqref{eq:LHSvalue}.
Therefore the minimum in \eqref{eq:RHSatmost1} cannot be \(0\).
Combining this with \eqref{eq:RHSatmost1}, we obtain
\begin{equation}\label{eq:RHSvalue}
\min_{P\in\Gl(6,K)}
\rank(PM_C-M_0P)
=
1.
\end{equation}

From \eqref{eq:LHSvalue} and \eqref{eq:RHSvalue}, we have
\[
\min_{X\in K^{3\times 3}}
\rank(AX-XB-C)
=
2
>
1
=
\min_{P\in\Gl(6,K)}
\rank(PM_C-M_0P).
\]
This proves that Conjecture \ref{con:LinWimmer} is false.

\section{Dimensional minimality over the complex field}\label{sec:minimality}

The counterexample in the preceding section has size \(m=n=3\).  
We now show that, over \(\mathbb C\), this size is minimal.

\begin{lemma}\label{lem:scalar-side}
Let \(K\) be any field.  
If either \(A\) or \(B\) is a scalar matrix, then
\[
\alpha(A,B,C)=\gamma(A,B,C).
\]
\end{lemma}

\begin{proof}
We first assume that \(A=\lambda I_m\).  
Then 
$$AX-XB-C
=
\lambda X-XB-C
=
-X(B-\lambda I_n)-C$$
and
\begin{align*}
PM_C-M_0P
=&
P\left(\lambda I_{m+n}+\begin{pmatrix}0&C\\0&B-\lambda I_n\end{pmatrix}\right)
-
\left(\lambda I_{m+n}+\begin{pmatrix}0&0\\0&B-\lambda I_n\end{pmatrix}\right)P\\
=&P\begin{pmatrix}0&C\\0&B-\lambda I_n\end{pmatrix}
-
\begin{pmatrix}0&0\\0&B-\lambda I_n\end{pmatrix}P.
\end{align*}
We have 
\begin{align*}
\min_{P,~Q\in\Gl(m+n,K)}&\rank\left(P\begin{pmatrix}0&C\\0&B-\lambda I_n\end{pmatrix}
-
\begin{pmatrix}0&0\\0&B-\lambda I_n\end{pmatrix}Q\right)\\
&\leq \min_{P \in\Gl(m+n,K)}\rank(PM_C-M_0P)\\
&\leq \min_{X\in K^{m\times n}}\rank(AX-XB-C)\\
&=\min_{X,~Y\in K^{m\times n}}\rank(-Y(B-\lambda I_n)-C).
\end{align*}
We may assume \(A=0\), in which case the result follows from \eqref{eq:LWequivalence}.

The proof for scalar \(B\) is similar.
\end{proof}

\begin{lemma}\label{lem:two-by-two-nonscalar}
Let \(T\in\mathbb C^{2\times2}\).  
If \(T\) is not a scalar matrix, then every eigenvalue of \(T\) is nonderogatory.
\end{lemma}

\begin{proof}
Let \(\lambda\) be an eigenvalue of \(T\).  
If \(\lambda\) has algebraic multiplicity one, then its geometric multiplicity is $1$, so \(\lambda\) is nonderogatory.  
If \(\lambda\) has algebraic multiplicity two, then the Jordan form at \(\lambda\) is either
\[
J_1(\lambda)\oplus J_1(\lambda)
~\text{or}~
J_2(\lambda).
\]
The first possibility gives \(T=\lambda I_2\), which is excluded.  
Hence the second possibility holds, and again geometric multiplicity one.  
Thus every eigenvalue of \(T\) is nonderogatory.
\end{proof}

\begin{proposition}\label{prop:minimal-dimension}
Over \(\mathbb C\), identity \eqref{eq:LWconjecture} holds whenever
$
\min\{m,n\}<3.
$
\end{proposition}

\begin{proof}
Suppose first that \(m<3\).  
If \(m=1\), then \(A\) is scalar, and Lemma~\ref{lem:scalar-side} applies.  

Now suppose \(m=2\).  
If \(A\) is scalar, Lemma~\ref{lem:scalar-side} again applies. 
If \(A\) is not scalar, Lemma~\ref{lem:two-by-two-nonscalar} shows that every eigenvalue of \(A\) is nonderogatory.  
Therefore every common eigenvalue of \(A\) and \(B\) is nonderogatory for at least one of the two matrices, namely for \(A\).  
By Theorem \ref{thm:FW-positive}, identity \eqref{eq:LWconjecture} holds. 

The case \(n<3\) is similar.
\end{proof}

\begin{remark}\label{rem:FW-not-necessary}
Theorem \ref{thm:FW-positive} gives sufficient conditions, not necessary ones.  
Lemma~\ref{lem:scalar-side} shows that equality holds whenever one of \(A,B\) is scalar, even though the nonderogatory hypothesis in Theorem \ref{thm:FW-positive} may fail and the semisimple-both-sides hypothesis in Theorem \ref{thm:FW-positive} may also fail.  
Thus the logical role of the result of Ferrante and Wimmer \cite{FerranteWimmer2013} is to provide broad spectral classes where equality is guaranteed, while the present counterexample shows that equality may fail outside those classes.
\end{remark}

\section{Conclusion}

We have exhibited a \(3\times 3\) counterexample over an arbitrary field to the rank-minimization conjecture of Lin and Wimmer for Roth's similarity theorem.
The example also indicates that the freedom to choose an arbitrary invertible matrix \(P\in\Gl(m+n,K)\) can reduce the rank below what can be attained by block upper triangular intertwiners of the form
\[
\begin{pmatrix}
I_m&X\\
0&I_n
\end{pmatrix}.
\]
By Proposition~\ref{prop:minimal-dimension}, over \(\mathbb C\) this phenomenon cannot occur when one of the two matrix sizes is less than \(3\).
Therefore the present \(3\times3\) construction is dimensionally minimal over the complex field.
A natural remaining problem is to characterize more precisely when the basic inequality \(\gamma(A,B,C)\le \alpha(A,B,C)\) from Proposition~\ref{prop:basic-comparison} is strict.

~\\

\noindent{\bf Acknowledgement.} The work was supported by the National Science Foundation of Anhui Higher Education Institutions of China (KJ2021ZD0058).

\end{document}